\documentclass{rmmcart}

  \usepackage{amsmath,amssymb,amsfonts,bbm,amsthm}
\usepackage{color,diagbox}
\usepackage{pdfpages}
\usepackage{graphicx}
\usepackage{geometry}
\usepackage{url}
\usepackage{enumerate}
\usepackage{extarrows,booktabs,multirow}

\theoremstyle{plain} \newtheorem{thm}{\bf Theorem}[section]
 \newtheorem{lem}[thm]{\bf Lemma}
  \theoremstyle{remark}

\newtheorem{rem}{Remark}

\newcommand{\Rg}       {{\hbox{I\kern-.22em\hbox{R}}}}
\newcommand{\Pg}       {{\hbox{I\kern-.22em\hbox{P}}}}
\newcommand{\Eg}       {{\hbox{I\kern-.22em\hbox{E}}}}

\newcommand{\cov}{\mathop{\mathrm{cov}}\nolimits}

\title{Simulation of Integro-Differential Equation and Application in Estimation of Ruin Probability with Mixed Fractional Brownian Motion}

\author{Chunhao Cai}
\address{School of Mathematics, Shanghai University of Finance and Economics, Shanghai, China}
 \email{caichunhao@mail.shufe.edu.cn}
 \thanks{The first author is the corresponding author}

\author{Weilin Xiao}
 \address{School of Management, Zhejiang University, Hangzhou, China}
 \email{wlxiao@zju.edu.cn}
\thanks{Weilin Xiao was supported by the National Natural Science Foundation of China, grant \# 71871202}

 \date{Month, Day, Year}

\keywords{integro-differential equation,mixed fractional Brownian motion, fundamental martingale, Girsanov theorem, estimation of ruin probability}
 \subjclass{Primary 60G22, Secondary 62F10, 65R20}

\begin{document}
 \begin{abstract}
 In this paper, we are concerned with the numerical solution of one type integro-differential equation by a probability method based on the fundamental martingale of mixed Gaussian processes.  As an application, we will try to simulate the estimation of ruin probability with an unknown parameter driven not by the classical L\'evy process but by the mixed fractional Brownian motion.  
\end{abstract}
\maketitle

\section{Introduction}
In this paper, we are concerned with the numerical solution of one type integro-differential equation by a probability method. Let us consider the following integro-differential equation:
\begin{equation}\label{eq: integro-differential}
g(s,t)+\frac{\partial}{\partial s}\int_0^t g(r,t)\frac{\partial}{\partial r} R(r,s)dr=1,\,\,\, 0<s\leq t \leq T,
\end{equation}
where $R(s,t)$ is defined by 
\begin{equation}\label{eq: covariance}
R(s,t)=\frac{1}{2}\left( t^{2H}+s^{2H}-|t-s|^{2H}\right),\,\, s,t\in [0,T].
\end{equation}
and $H\in (0,1)$ but $H\neq 1/2$. From \cite{CCK}, we can rewrite this equation for different $H$:
\begin{enumerate}
\item[$\bullet$]  For $1/2<H<1$, when the derivative in \eqref{eq: integro-differential} can be interchanged with integration, so that it takes the form of the Wiener-H\"opfner equation 
\begin{equation}\label{eq: integral big}
g(s,t)+H(2H-1)\int_0^t g(r,t)|r-s|^{2H-2}dr=1.
\end{equation}

\item[$\bullet$]  For $0<H<1/2$, the derivative and the integration are no longer interchangeable and the integral equation \eqref{eq: integral big} makes no sense. However, using tools from fractional calculus, \eqref{eq: integro-differential} can still be solved by reduction to a different weakly singular integral equation (see \textit{e.g.} \cite{CCK}):
\begin{equation}\label{eq: integral small}
g(s,t)+\beta_Ht^{-2H}\int_0^t g(r,t)K^H\left(\frac{r}{t},\frac{s}{t}\right)dr=c_Hs^{1/2-H}(t-s)^{1/2-H}
\end{equation}
with the kernel 
$$
K^H(u,v)=|u-v|^{-2H}N(u,v)
$$
where $N(u,v)\in C([0,1]^2)$ is defined in the equation (5.6) of \cite{CCK} and $\beta_H$, $c_H$ are two constants.
\end{enumerate}

\begin{rem}\label{rem1}
As presented in \cite{CCK}, we can extend $R(s,t)$ of \eqref{eq: covariance} to the covariance function of a general centered Gaussian process $G=(G_t,\, 0\leq t \leq T)$ whose kernel
$$
K(s,t)=\frac{\partial ^2}{\partial s \partial t} R(s,t)=\frac{\partial ^2}{\partial s \partial t}\mathbf{E}G_tG_s,\,\, s\neq t
$$
has a weakly singularity on the diagonal.  For kernels with stronger singularity as $H<1/2$ in \eqref{eq: integro-differential} , our method presented in Section \ref{sec:simulation} is applicable to processes with additional "fractional" structure such as sub-fractional Brownian motion, bifractional Brownian motion, Riemann-Liouville process.
\end{rem}

\begin{rem}
Note that the values of $g(s,t)$ of the equations \eqref{eq: integral big} and \eqref{eq: integral small} on the sub-diagonal $\{0<s<t\leq T\}$ determine $g(s,t)$ on the super-diagonal. Hence the problem of solving \eqref{eq: integro-differential} reduces to solving it one the sub-diagonal $\{0<s<t\}$ for all $t\in [0,T]$. For a fixed $t\in [0,T]$, the restriction of \eqref{eq: integral big} and \eqref{eq: integral small} to the sub-diagonal is the Fredholm equation of second kind, whose solvability is very well known under various conditions (see, \textit{e.g.} \cite{Kress14}).
\end{rem}

If we only consider the weakly singular Fredholm integral equation of second kind, there are many different methods to obtain the numerical solution. For example, based on the fact that the equations usually have singularities in their derivatives, reflecting the singularity of the kernel, a product integration method \cite{Schneider81} , a collocation method\cite{VP81}  and a Galerkin method \cite{Graham82} have been designed, Cao and Xu developed in \cite{CX94} a singularity preserving Galerkin method.   Recently, Parts et \textit{a.l.} \cite{PPT05} investigated  a piecewise polynomial collocation method with graded meshes and Cao et \textit{a.l.} \cite{CHLX07} established the super-convergence property of the iteration of  the Hybrid collocation method. We can see \cite{Albert56} and \cite{Atkinson97} for details of other methods. 

However, when $H<1/2$ for the equation \eqref{eq: integro-differential},  as presented in Remark \ref{rem1} we have a kernel with stronger singularity and the previous mentioned numerical methods can not be used directly if we do not consider the formula of the equation \eqref{eq: integral small}. That is why we want to use a probability method to find the numerical solution, we will see that this method is easier to understand and utilize. In fact, this probability method is based on the the definition of the fundamental martingale of mixed fractional Brownian motion which will be presented in Section \ref{sec:simulation}, so first of all, let us introduce the mixed fractional Brownian motion (mfBm for short) denoted by $\xi=(\xi_t,\, t\in [0,T])$, i.e. the process
$$
\xi_t=W_t+B_t^H,\, t\in [0,T]
$$
where $W=(W_t,\, 0\leq t\leq T)$ is the standard Brownian motion and $B^H=(B_t^H,\, 0\leq t\leq T)$ is an independent fractional Brownian motion with Hurst parameter $H\in (0,1)$, that is , a zero mean Gaussian process with the covariance function
$$
\mathbf{E}B_t^H B_s^H=R(s,t)=\frac{1}{2}(t^{2H}+s^{2H}-|t-s|^{2H}),\, s,t \in [0,T].
$$ 

The interest in this process was triggered by Cheridito \cite{Cheridito}, in which the author discovered a curious change in the properties of $\xi$ occurring at $H=3/4$. In order to develop the theory of \cite{Cheridito} and find the maximum likelihood estimator of the drift parameter of mixed fractional Ornstein-Uhlenbeck process \cite{CK17}, Cai et \textit{a.l.} suggested the \cite{CCK}  the canonical representation of mfBm based on the fundamental martingale
\begin{equation}\label{eq: martingale}
M_t:=\mathbf{E}(W_t|\mathcal{F}_t^{\xi}),\,\, t\in [0,T]
\end{equation}
where $\mathcal{F}^{\xi}=\{\mathcal{F}_t^{\xi},\, t\in [0,T]\}$ is the filtration generated by the mfBm. By the Theorem 2.4 in \cite{CCK}, this martingale satisfies
\begin{equation}\label{eq: martingale 1}
M_t=\int_0^t g(s,t)d\xi_t,\,\, \langle M\rangle_t=\int_0^t g(s,t)ds,\, t\in[0,T]
\end{equation}
where $\langle M\rangle$ is the quadratic variation of the martingale $M$ and the function $g(s,t)$ is the unique solution \eqref{eq: integro-differential}. When the function $g(s,t)$ appears in \eqref{eq: martingale 1}, our first goal in this paper is to present a numerical solution for $g(s,t)$ when $t$ fixed with the formula of the conditional expectation \eqref{eq: martingale}.

The second contribution of this paper is the estimation of ruin probability driven by mfBm using this numerical solution. Let us observe that mfBm besides captures the fluctuation of financial assets (see, for example, \cite{Cheridito}), it can also explain the phenomenon of the insurance surplus model. In ruin theory, the classical risk model plays the central role in the theoretical analysis in ruin theory, and lots of nice results have been obtained by actuarial researchers, (for example, see \cite{Dufresne}, \cite{Asmussen}). In the classical risk model, the insurance surplus is always supposed as
$$
X_t=u+c t+\sigma W_t-\sum_{i=1}^{N_t}U_i\,,
$$
where $u$ is a constant and $\sum\limits_{i=1}^{N_t}U_i$ is considered as a compound poisson process with a fixed intensity $\lambda$. In contrast to the classical, independent, identically distributed assumptions, we are interested in the case where $(U_k: k\in \mathbb{N})$ are stongly dependent. In queueing systems, the relevance of such dependence assumptions is currently receiving much attention. We shall construct a sequence of risk processes with the strongly dependence  and show that it converges to a surplus process driven by mfBm.  It should be noted that there are many works concentrated on the convergence to fractional Brownian motion such as 
\cite{Taqqu,Natha, Bill} and here we take  \cite{Michna98a} for example to explain how it works and what is the unknown parameter.

Now let us  consider a sequence of surplus process $X^{(n)}=(X_t^{(n)})_{t\geq 0}$ given by
$$
X_t^{(n)}=u^{(n)}+c^{(n)}t-\sigma^{(n)} W_t-\sum_{k=1}^{N_t^{(n)}}U_k^{(n)}\,,
$$
where $x^{(n)}>0$, $c^{(n)}>0$ denote the initial risk reserve and the premium rate, $N_t^{(n)}$ is the sequence of Poisson process with intensity $n$ and $\sigma^{(n)}>0$ . We assume that the claims are of the form $U_k^{(n)}=\frac{1}{\varphi(n)}U_k$ where $(U_k,k\in \mathbb{N})$ is a stationary sequence with common distribution $F$ and mean $\mu$ such that:
$$
\frac{1}{\varphi(n)}\sum_{k=1}^{[nt]}(U_k-\mu)\Longrightarrow B_t^H\,,
$$
where $B_t^H$ is a fractional Brownian motion with Hurst parameter $H>1/2$ and the symbol $``\Longrightarrow"$ denotes the weak convergence in Skorokhod topology (see \textit{e.g.} \cite{Bill}). Here we suppose $\varphi(n)=n^HL(n)$ and the function $L$ is slowly varying at infinity.  By the Theorem 3 of \cite{Michna98a} we know that:
\begin{lem}\label{cov to model}
Let $N_t^{n}$ be a sequence of Poisson process with intensity n such that
$$
\frac{N_t^{(n)}-nt}{\varphi(n)}\rightarrow 0\,,
$$
in probability in the Skorokhod topology. Assume also that
$$
\lim_{n\rightarrow\infty}\left(c^{(n)}-\frac{n\mu}{\varphi(n)}\right)=\vartheta\,,
$$
and
$$
\lim_{n\rightarrow \infty}u^{(n)}=u,\, \lim_{n\rightarrow \infty}\sigma^{(n)}=\sigma\,.
$$
Then
\begin{equation}\label{eq: cov to model}
X_t^{(n)}=u^{(n)}+c^{(n)}t-\sigma^{(n)}W_t-\sum_{k=1}^{N_t^{(n)}}U_k^{(n)}\Longrightarrow u+\vartheta t-\sigma W_t-B_t^H\,,
\end{equation}
in the Skorokhod topology as $n\rightarrow \infty$.
\end{lem}
From now on, we consider the following surplus process 
\begin{equation}\label{eq: surplus}
X_t=u+\vartheta t-\xi_t,\, t\geq 0\,,
\end{equation}
where
$$
\xi_t=\sigma W_t+B_t^H.
$$
Let us mention that the main properties of $\xi$ will not depend on the parameter $\sigma$, consequently, for further study we always assume $\sigma=1$ and now the process $\xi$ is the mixed fractional Brownian motion. Obviously, in Lemma \ref{cov to model}, the parameters $H$ and $\vartheta$ are unknown because the function $L$ is unknown. For $H$, there exist so many works about the estimation such as  the general Maximum Likelihood estimation, Whittle's method \cite{Jan} and the power variations especially the quadratic variation \cite{Mishura}. Consequently, without loss of generality, for further statistical analysis, we assume that $H$ is known and always $H>1/2$. The main work of our article will be the estimation of the ruin probability with the unknown parameters using the past surplus data and study its asymptotic properties using the Delta method. We define the finite time ruin probability with $\psi(u,T)$ in the time interval $[0,T]$:
$$
\psi(u,T)=\mathbf{P}\left(\inf_{t\in [0,T]}X_t<0\right)\,.
$$
We can state it by another way
\begin{equation}\label{ruin stopping time}
\psi(u,T)=\mathbf{P}(\tau\leq T)=\mathbf{E}\left[\mathbf{1}_{\left\{\left[\sup\limits_{0\leq t\leq T}( \xi_t- \vartheta t)\right]>u\right\}} \right]\,,
\end{equation}
where $\tau=\inf\{t>0|X_t<0\}$ is the time of ruin for surplus $X$. For $T=\infty$ it is called ultimate ruin probability and we will use $\psi(u)$ replacing $\psi(u, T=\infty)$. When using the Delta method to study the asymptotic properties of $\psi(u)$ which depends on $\vartheta$, we have to deal with the problem of the estimation of $\vartheta$ and the formula of $\partial_{\vartheta}\psi(u)$. In Section \ref{sec:Girsanov} we will see that the ML estimator of $\vartheta$ depends on the function $g(s,t)$, that is why we have to present the simulation of this function.

The rest of the paper is organized as follows. In Section \ref{sec:simulation}, we consider the problem of simulating the function of $g(s,t)$. In section \ref{sec:Girsanov}, we present the Girsanov formula for the mfBm. The asymptotic behaviour of $\eta_a(v)$ defined in \eqref{eq: eta a} and the problem of estimating $\vartheta$ are also discussed in this section. The proofs of some main results will be presented in Section \ref{sec: Appendix}.

\section{Simulation of the integro-differential equation}\label{sec:simulation}
\subsection{Some spaces and operators}
When the function $g(s,t)$ appears in the equation \eqref{eq: martingale 1}, the fundamental martingale --a conditional expectation, the numerical method of course will be based on the $L^2$ projection. Throughout we assume that all the random variables and stochastic processes are supported on a probability space $(\Omega, \mathcal{F}, \mathbf{P})$. Now for  $f:[0,t]\rightarrow \mathbb{R}$, define the operators 
$$
(\Psi f)(s,t)=-2H\frac{d}{ds}\int_s^t f(r) r^{H-1/2}(r-s)^{H-1/2}dr,\, 0\leq s\leq t
$$
and the space
$$
\Lambda_t^{H-1/2}:=\left\{f:[0,t]\rightarrow \mathbb{R}: \int_0^t(s^{1/2-H}(\Psi f)(s,t))^2ds<\infty,         \right\}
$$
with the scalar product
$$
\langle f,g \rangle_{\Lambda_t^{H-1/2}}:= \frac{2-2H}{\lambda_H}\int_0^t s^{1-2H}(\Psi f)(s,t)(\Psi g)(s,t)ds.
$$
From Lemma 3.9 in \cite{CCK} we know that for fixed $t$ the function $g(\cdot, t)\in L^2([0,t])\cap \Lambda_t^{H-1/2}$ then the stochastic integral $\int_0^t g(s,t)d\xi_s$ can be defined in a usual way presented in \cite{Taqqu}. When $\xi$ is a Gaussian process and with the Normal Correlation Theorem we have 
$$
\mathbf{E}\left(W_t-\int_0^t g(s,t)d\xi_s\right)\xi_t=0.
$$

\begin{rem}
To find the exact formula of $g(s,t)$ when $H<1/2$, we have to choose the stochastic integral with respect to $\xi$ with an arbitrary test function as the integrand. In this way, the expectation 
$$
\mathbf{E}\int_0^t f(s)dB_s^H\int_0^t g(s)dB_s^H\neq H(2H-1)\int_0^t \int_0^t f(u)g(v)|u-v|^{2H-2}dudv
$$
and the details will be  explained in \cite{CCK}. But for the simulation, we only need to choose $\xi_t$ or the constant function 1 as the test function, so the only condition needed is $g\in L^2([0,t])\cap \Lambda_t^{H-1/2}$.
\end{rem}
\subsection{Numerical approximation}
To illustrate the procedure of the simulation, we need the following Lemma:
\begin{lem}\label{th: simu}
For t fixed we define $t_i=ti/2^n,\, i=0,\cdots, 2^n$ and $\mathcal{F}^{\xi}_{t,n}=\sigma \{\xi_{t_i}-\xi_{t_{i-1}},\, i=1.\cdots,2^n\}$ then 
\begin{equation}\label{eq: N1}
\mathbf{E}(W_t|\mathcal{F}_{t,n}^{\xi})=\sum_{i=1}^{2^n}g_{i-1}^n(\xi_{t_i}-\xi_{t_{i-1}})
\end{equation}
with constants $g_{i-1}^n,\, i=1,\cdots, 2^n$. Moreover, if we define the equation 
\begin{equation}\label{eq: N2}
g_n(s,t):=\sum_{i=1}^{2^n}g_{i-1}^n\mathbf{1}_{\{s\in[t_{i-1},t_i)\}}
\end{equation}
then $g_n(s,t)\xrightarrow {n\rightarrow \infty}g(s,t)$ with $g(s,t)$ the unique solution of \eqref{eq: integro-differential}.
 \end{lem}
 \begin{proof}
 When $\mathbf{E} W_t=0$, the equation \eqref{eq: N1} can be easily obtained by the normal correlation theorem. Following the arguments of the proof of Lemma 10.1 in \cite{LS01}, $\mathcal{F}_{t,n}^{\xi}\nearrow\mathcal{F}_t^{\xi}$ and by the martingale convergence
 $$
 \lim_n\mathbf{E}(W_t|\mathcal{F}_{t,n}^{\xi})=\mathbf{E}(W_t|\mathcal{F}_t^{\xi})=\int_0^t g(s,t)d\xi_s,\,\, \mbox{in} \\\ L^2(\Omega,\mathbf{P})
 $$
 then the convergence of \eqref{eq: N2} can be proved with the same way of Lemma 3.9 of \cite{CCK}.
\end{proof}
To simplify the procedure, we just divide the interval $[0,t]$ with n parts and denote $\Delta \xi_i=\xi_{t_i}-\xi_{t_{i-1}}$ with $t_i=it/n$, then the conditional expectation \eqref{eq: N1} can be presented by 
$$
\mathbf{E}(W_t|\Delta \xi_1,\,\cdots,\, \Delta \xi_n)=\sum_{i=1}^n\gamma_i \Delta \xi_i.
$$
The vector $\gamma=\left(\begin{array}{ccc}\gamma_1 & \cdots & \gamma_n\end{array}\right)$  is the solution of  the linear equations
$$
\mathbf{E}(W_t-\sum_{i=1}^n\gamma_i \Delta \xi_i)\Delta \xi_i=0,\, i=1,2,\cdots,n. 
$$
We can easily obtain
$$
\gamma=\frac{t}{n}\mathbf{1}^* R_n^{-1}
$$
where $R_n$ is a n-order symmetric matrix with 
\begin{equation}\label{eq: partial derivative}
R_n^{i,j}=\left\{
\begin{array}{l}
t/n+(t/n)^{2H},\,\, i=j\\
\frac{1}{2}(t/n)^{2H}\left(|i-1-j|^{2H}+|j-1-i|^2H-2|i-j|^{2H}\right),\,\, others\\
\end{array}
\right.
\end{equation}
and $\mathbf{1}^*=\left(\begin{array}{ccc}1 & \cdots & 1\end{array}\right)$. Obviously, the vector $\gamma$ is our numerical solution of $g(s,t)$ for $t$ fixed. 

\subsection{The simulation results for long memory and rough path}
As we know, when $H>1/2$, the fractional Brownian motion presents the properties of long memory and for $H<1/2$ the properties of rough path. We take two different $H$ to present the simulation results of the function $g(s,t)$. In the following simulations we fix $t=1$ and the distance of the function will be $d=1/3000$.

\begin{center}
\resizebox{150mm}{70mm}{\includegraphics{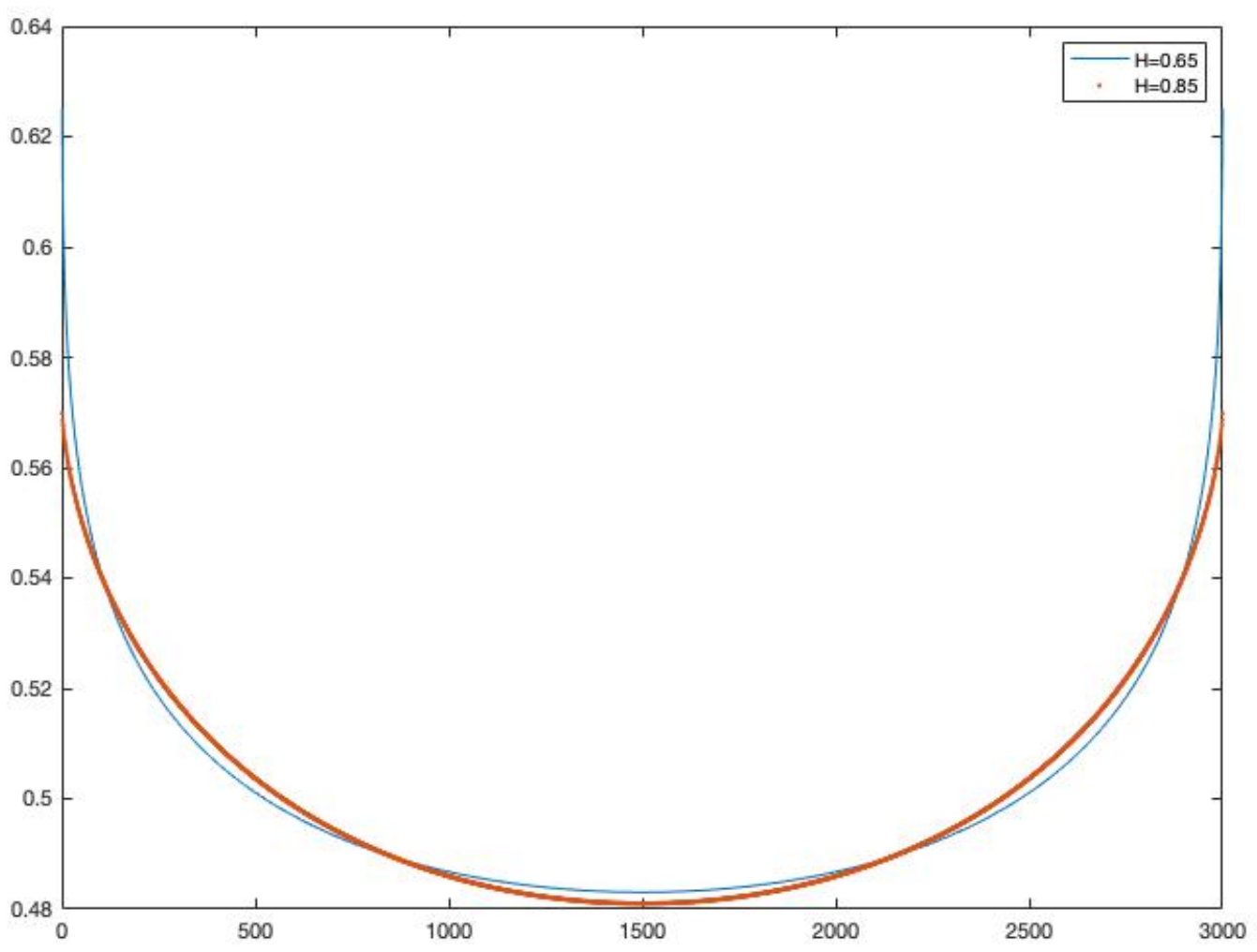}}\\
Fig.1. For $H>1/2$.
\end{center}

\begin{center}
\resizebox{150mm}{70mm}{\includegraphics{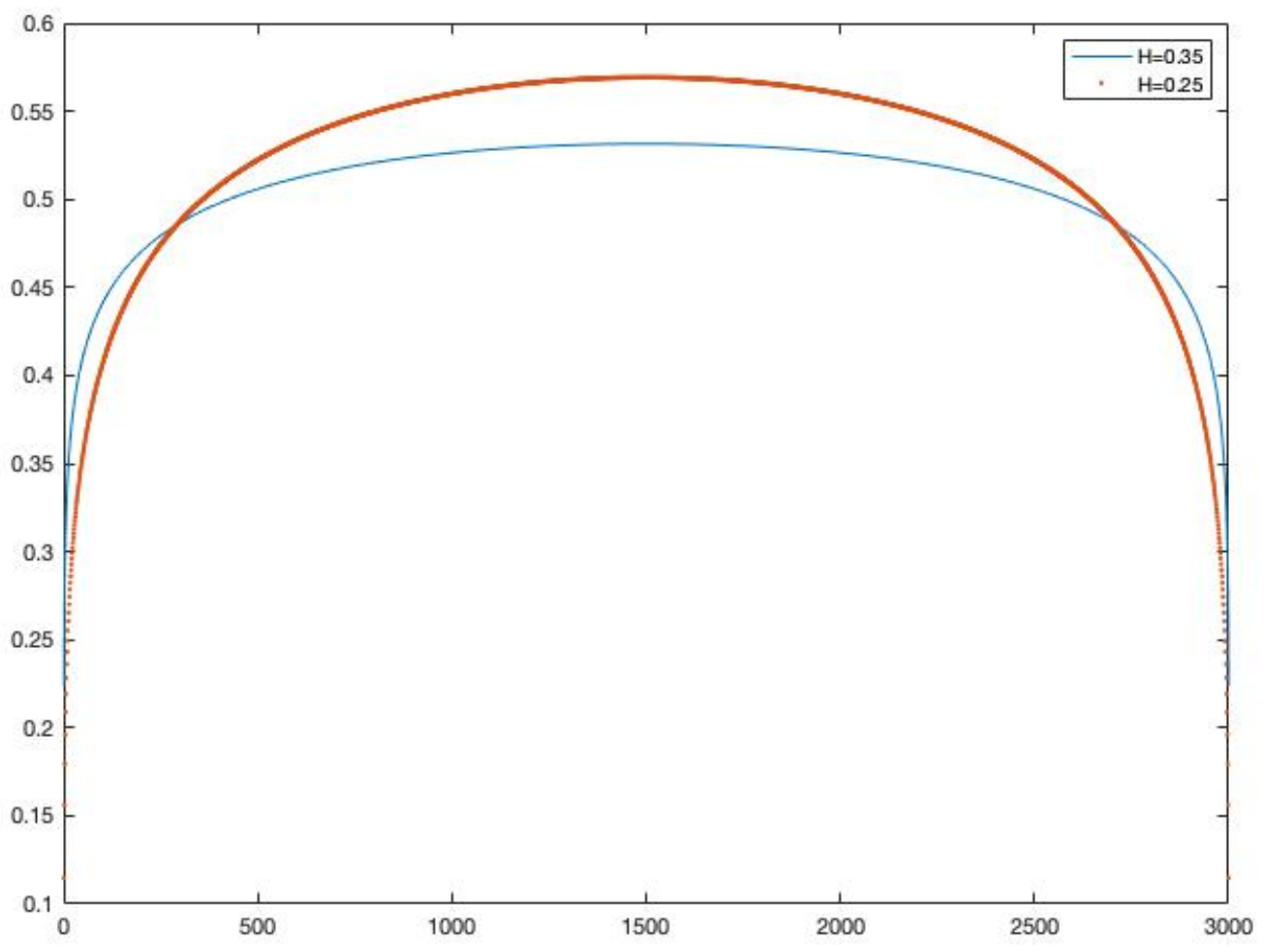}}\\
Fig.2. For $H<1/2$.
\end{center}

Is our simulation reasonable or not, How to verify it? Obviously, we have no possibility to find the analytical solution even for $H>1/2$, so the simplest way to compare with the theoretical solution is not feasible and  in order to check this simulation reasonable or not we need another tool. Fortunately, in the procedure of the drift estimation of linear regression model (see the last part of \cite{CK}) which will also be used later, we have the following Lemma:
\begin{lem}\label{Lemma limit performance}
For $H\in (0,1)$ and $H\neq 1/2$, If the function $g(t,T),\, 0\leq t\leq T$ is the unique solution of the integro-differential equation \eqref{eq: integro-differential} that is 
$$
g(t,T)+H\frac{d}{dt}\int_0^T g(r,T)|r-t|^{2H-1}=1,\, 0\leq t \leq T
$$
then the limit behavior of quadratic variation $\langle M\rangle_T=\int_0^T g(t,T)dt$ defined in \eqref{eq: martingale 1} differs from $H$:
\begin{itemize}
\item for $H>1/2$,
\begin{equation}\label{eq: asym variance bigger}
T^{2-2H}\frac{1}{\langle M\rangle_T} \xrightarrow{T\rightarrow \infty}\frac{2H\Gamma(H+1/2)\Gamma(3-2H)}{\Gamma(3/2-H)}.
\end{equation}
\vspace{1mm}
\item for $H<1/2$,
\begin{equation}\label{eq: asym variance smaller}
\frac{T}{\langle M\rangle_T}\xrightarrow {T\rightarrow \infty}1.
\end{equation}
\end{itemize}
\end{lem}
\begin{proof}
In \cite{CK}, the authors use the limits of the eigenvalues and eigenfunctions of the covariance operator of fractional Brownian motion to get the results. In fact we have an easy way that is for $H>1/2$, 
$$
T^{2-2H}\frac{1}{\langle M\rangle_T}=\frac{1}{\int_0^1 g_{\varepsilon}(u)du}\xrightarrow{\varepsilon \rightarrow 0} \frac{1}{\int_0^1g(u)du}
$$
where $g_{\varepsilon}(u)$ is the solution of 
$$
\varepsilon g_{\varepsilon} (u)+H(2H-1)\int_0^1 g_{\varepsilon}(v)|u-v|^{2H-2}dv=1,\, u\in [0,1]
$$
and $g(u)$ is the unique solution of limit equation of $g_{\varepsilon}(u)$ when $\varepsilon\rightarrow 0$:
$$
H(2H-1)\int_0^1 g(u)|u-v|^{2H-2}du=1.
$$
The similar proof can be achieved for $H<1/2$.
\end{proof}

From Lemma \ref{Lemma limit performance}, we use the numerical solution of Lemma \ref{th: simu} and show the performance of \eqref{eq: asym variance bigger} and \eqref{eq: asym variance smaller} for $H=0.65$ and $H=0.35$ in the following Table. We can see that when $T$ is larger the simulation is closer to the theoretical result.

\begin{center}
Table 1. Simulation results for different values of $H$ and $T$ of Lemma \ref{Lemma limit performance}.
\vspace{1mm}\\
\begin{tabular}{|@{}c|c|c|c|c|}
\hline
\diagbox[width=5em,trim=l]{$H$}{$T$} & 200 & 1000 & 5000 & limitation\\
\hline
0.65 & 0.9701 & 0.9812 & 0.9869 & 0.9907\\
\hline
0.35 & 1.1080  & 1.0605 & 1.0328 & 1  \\
\hline
\end{tabular}
\end{center}

\begin{rem}
From the previous table we are confident that our method of simulation will adapt the real values of function \eqref{eq: integro-differential}. However, the convergence rate is still an open problem and we leave it for further study.
\end{rem}

\section{Girsanov Formula and finite-time ruin probability}\label{sec:Girsanov}

\subsection{Girsanov formula}
Now we will introduce the Girsanov formula for the mfBm. Let  $\Xi_t(M)$  be the stochastic exponent of M:
$$
\Xi_t(M):=\exp \left(M_t-\frac{1}{2}\langle M, M\rangle_t \right)\,.
$$
Thus it is a martingale with mean 1. For $a\in \mathbb{R}$, let $\mathbf{P}_a$ be a probability on $(\Omega, \mathcal{F}, (\mathcal{F}_t)_{t\geq 0})$ defined by
\begin{equation}\label{Pa}
\frac{d\mathbf{P}_a}{d\mathbf{P}}=\Xi_T(aM)
\end{equation}
on $\mathcal{F}_T$. As in \cite{NVV99}, we will prove that $\Xi_T(aM)$ is in fact the likelihood ratio between the following two hypotheses:

$H_0$: With respect to the measure $\mathbf{P}$, the process $X$ is a mfBm with Hurst parameter $H$,  i.e.  $\bar{X}_t=\xi_t$.

$H_a$: With respect to the measure $\mathbf{P}_a$, the process X is a mfBm with constant drift a, i.e. $\bar{X}_t=\xi_t+at$.

\begin{thm}\label{Girsanov formula}
With respect to the measure $\mathbf{P}_a$, the process $\xi$ is a mfBm with drift a, i.e. the distribution of $\xi$ with respect to $\mathbf{P}_a$ is the same as the distribution of $\xi_t+at$ with respect to $\mathbf{P}=\mathbf{P}_0$.
\end{thm}

\begin{rem}
Theorem \ref{Girsanov formula} tells us  that when $(\xi_t)_{0 \leq t\leq  T}$ is a mfBm with parameter $H$ under $\mathbf{P}$, then the process $\xi_t(a):= \xi_t-at$, $0\leq t \leq T$ is a mfBm under $\mathbf{P}_a$.
\end{rem}

\subsection{Finite-time and Ultimate ruin probability}
We will rewrite the finite ruin probability \eqref{ruin stopping time} with the formula of expectation. To achieve this goal, we define a stopping time $\eta_a(v)=\inf \{t>0|\xi_t+at>v\}$, it can be rewritten by 
\begin{equation}\label{eq: eta a}
\eta_a(u)=\inf \{t>0,\xi_t+(\vartheta+a)t-\vartheta t>u\},
\end{equation}
where due to Theorem \ref{Girsanov formula} the process $\xi_t+(\vartheta+a)t$ is a mfBm under the probability $\mathbf{P}_{-(\vartheta+a)}$. Moreover, a standard calculation yields
\begin{equation*}
\begin{aligned} % requires amsmath; align* for no eq. number
\psi(u,T) & = \mathbf{P}(\tau \leq T) \\
   & = \mathbf{E}_{-(\vartheta+a)}\left[\mathbf{1}_{\eta_a(u)\leq T}\right]\\
   & = \mathbf{E}\left[ \Xi_{T}(-(a+\vartheta)M) \mathbf{1}_{\eta_a(u)\leq T}  \right]\\
   & = \mathbf{E}\left[\mathbf{E}\left[ \Xi_T(-(a+\vartheta)M)|\mathcal{F}_{\eta_a(u)\wedge T}             \right] \mathbf{1}_{\eta_a(u)\leq T}\right]\\
   & =\mathbf{E}\left[\Xi_{\eta_a(u)}(-(a+\vartheta)M)\mathbf{1}_{\eta_a(u)\leq T}\right].
\end{aligned}
\end{equation*}
Thus, for any $a\in \mathbb{R}$, we have
\begin{equation}\label{finite ruin exact}
\psi(u,T)=\mathbf{E}\left[\exp\left(-(a+\vartheta)\int_0^{\eta_a(u)} g(s,\eta_a(u))d\xi_s-\frac{(a+\vartheta)^2}{2}\int_0^{\eta_a(u)}g(s,\eta_a(u))ds    \right) \mathbf{1}_{\{\eta_a(u)\leq T\}}\right]\,.
\end{equation}
In the case of $H>1/2$ when $\int_0^t g(s,t)ds=\int_0^t g^2(s,s)ds$ we have
\begin{equation}\label{finite ruin exact 1}
\psi(u,T)=\mathbf{E}\left[\exp \left(-(a+\vartheta)\int_0^{\eta_a(u)}  g(s,\eta_a(u))d\xi_s-\frac{(a+\vartheta)^2}{2}\int_0^{\eta_a(u)}g^2(s,s) ds \right)\mathbf{1}_{\{\eta_a(u)\leq T\}}\right]
\end{equation}
For the $u>0$ fixed, we have
$$
\lim_{a\rightarrow \infty}\mathbf{1}_{\{\eta_a(u)\leq T\}}=1,\,a.s
$$
Thus for  large enough $a$, it will be better to simulate the $\psi(u,T)$ by the Monte Carlo procedure. In particular, letting $T\rightarrow \infty$ for  both sides of \eqref{finite ruin exact} , the monotone convergence theorem yields an expression of ultimate ruin probability: for any $a>0$,
$$
\psi(u)=\mathbf{E}\left[\exp\left(-(a+\vartheta)\int_0^{\eta_a(u)} g(s,\eta_a(u))d\xi_s-\frac{(a+\vartheta)^2}{2}\int_0^{\eta_a(u)}g(s,\eta_a(u))ds    \right) \right]\,.
$$

\subsection{Asymptotic Properties of $\eta_a(v)$}\label{section 3}
We have got the exact expression of the ruin probability, but we have no efficient information about $\eta_a(v)$. What it will perform in the ruin? In this part, we will discuss the asymptotic behaviour of $\eta_a(v)$ for a large $v$ fixed when $H>1/2$. First of all, we have the upper bound of the expectation of $\eta_a^n(v)$.

\begin{thm}\label{upper bound}
Let $a>0$ and $\frac{1}{2}<H<1$ and $n\in \mathbb{N}$. Then for all $v>0$, we have
$$
\mathbf{E}\eta_a^n(v)\leq \frac{1}{\sqrt{2\pi}}\left(\frac{vH'}{\sqrt{2}(n-H')}L_n(v/\sqrt{2},H',a/\sqrt{2})+\frac{(1-H')a}{\sqrt{2}(n+1-H')} L_{n+1}(v/\sqrt{2},H',a/\sqrt{2}) \right)\,,
$$
where
\begin{equation}\label{Ln}
L_n(u,H,a)=\int_0^{\infty}\exp\left\{-\frac{1}{2}\left(ut^{-H/(n-H)}-at^{(1-H)/(n-H)}\right)^2\right\}dt
\end{equation}
and  $H'=\frac{2H+1}{4}$.
\end{thm}

The following theorem gives the limit performance of the expectation of $\eta_a^n(v)$ and the covergence of $\eta_a(v)$ in $L^n(\Omega)$.
\begin{thm}\label{asymptotic convergence}
For all $n\geq 1$ and $a>0$, we have
\begin{equation}\label{limit expectation}
\lim_{v\rightarrow\infty}\mathbf{E}\eta_a^n(v)\left(\frac{v}{\sqrt{2}}\right)^{-n}=\left(\frac{a}{\sqrt{2}}\right)^{-n}\,.
\end{equation}
Furthermore, we obtain
\begin{equation}\label{no expectation}
\lim_{v\rightarrow \infty}\eta_a(v)\left(\frac{v}{\sqrt{2}}\right)^{-1}=\left(\frac{a}{\sqrt{2}}\right)^{-1}
\end{equation}
in $L^n(\Omega)$. Finally, we have
\begin{equation}\label{Ln convergence}
\lim_{v\rightarrow\infty}\int_{\Omega}\left|\eta_a(v)\left(\frac{v}{\sqrt{2}}\right)^{-1}-\left(\frac{a}{\sqrt{2}}\right)^{-1}     \right|^nd\mathbf{P}=0.
\end{equation}
\end{thm}

\begin{rem}
Until now, the asymptotical performance for the situation $H<1/2$ is still not clear, it will be our future research interest.
\end{rem}

\subsection{Estimation of the ruin probability with the unknown parameter $\vartheta$}
At the beginning of this article, we have elaborated that with the approximation the parameter $\vartheta$ is unknown. In this part, we will try to estimate the $\vartheta$ and the ruin probability with the unknown parameter from the observation $(X_t)_{0\leq t \leq T'}$.  First we try to use the continuous time observation. Thus, in the model
$$
X_t=u+\vartheta-\xi_t,\, 0\leq t\leq T'\,.
$$
We can observe all the trajectories of $(X_t)_{0\leq t \leq T'}$, that will be the continuous-time observation method. Now let $Y_t=X_t-u$, then the process $(Y_t)_{0\leq t\leq T'}$ satisfies
$$
Y_t=\vartheta t-\xi_t\,.
$$
From this observation, we take $\hat{\vartheta}_{T'}$ as the Maximum Likelihood Estimator (MLE) of the unknown parameter $\vartheta$. At the same time with this MLE, we take $\hat{\psi}(u,T)=\psi(\hat{\vartheta}_{T'},u,T), T'<T$ as the ruin probability $\psi(u,T)$ (Finite-time ruin probability or ultimate ruin probability) with the unknown parameter $\hat{\vartheta}_{T'}$. As we know, the asymptotical properties of $\hat{\vartheta}_{T'}$ depends on $T'\rightarrow \infty$ so when we want to study this properties of $\hat{\psi}(u,T)$ we will take $T=\infty$ and take this notation as $\hat{\psi}(u)$ then we have the following results:
\begin{thm}\label{convergence ruin}
When $\hat{\psi}(u)$ defined as the ultimate ruin probability with respect to the estimator $\hat{\vartheta}_{T'}$, then 
\begin{itemize}
\item for $H>1/2$, 
\begin{equation}\label{eq: bigger asym}
T'^{1-H}\left(\hat{\psi}(u)-\psi(u)\right)\sim \mathcal{N}\left(0,\,  \frac{2H\Gamma(H+1/2)\Gamma(3-2H)}{\Gamma(3/2-H)}\partial_{\vartheta}^2\psi(u)       \right)
\end{equation}
\item for $H<1/2$,
\begin{equation}\label{eq: smaller asym}
\sqrt{T}\left(\hat{\psi}(u)-\psi(u)\right)\sim \mathcal{N}\left(0, \partial_{\vartheta}^2\psi(u) \right)
\end{equation}
where $\Gamma(\cdot)$ is the Gamma function.
\end{itemize}
\end{thm}

\begin{rem}
From the equation \eqref{finite ruin exact}  the function $\psi(u,T)$ is continuous differentiable  with respect to $\vartheta$ for $a$ fixed and when $H>1/2$ we have
\begin{eqnarray*}
\partial_{\vartheta}\psi(u)&=&\mathbf{E}\left[-\frac{\vartheta}{\sigma}\int_0^{\eta_a(u)}  g(s,\eta_a(u))d\xi_s-(a+\vartheta)\int_0^{\eta_a(u)}g^2(s,s) ds\right.\\
&&\left.\exp \left(-\frac{a+\vartheta}{\sigma}\int_0^{\eta_a(u)}  g(s,\eta_a(u))d\xi_s-\frac{(a+\vartheta)^2}{2}\int_0^{\eta_a(u)}g^2(s,s) ds \right)\right]
\end{eqnarray*}
In fact, we can also use Malliavin calculus to define this derivative which will be presented in another research work. On the other hand, we want to emphasize that in the simulation, when the equation \eqref{finite ruin exact} has the exponential function, so the tail of Gaussian distribution of $\xi_t$ for every  $t$ can not be ignored.  
\end{rem}
\begin{proof}
From \cite{CK} we know that the MLE of $\vartheta$ is given by
\begin{equation}\label{eq: MLE}
\hat{\vartheta}_{T'}=\frac{\int_0^{T'}g(t,T')dY_t}{\int_0^{T'}g(t,T')dt}
\end{equation}
and the corresponding estimation error is normal
\begin{equation}\label{eq: different distribution}
\hat{\vartheta}_{T'}-\vartheta \sim \mathcal{N} \left(0, \frac{1}{\int_0^{T'} g(t,T')dt} \right)
\end{equation}
where the function $g(t,T'),\, t\in [0,T']$ is the solution of the equation \eqref{eq: integro-differential}. Here   $\sim$ is not the asymptotical symbol but for every $T'$ satisfied. From the two limit properties \eqref{eq: asym variance bigger} and \eqref{eq: asym variance smaller} we have
$$
\lim_{T'\rightarrow \infty} T'^{2-2H} \mathbf{E}_{\vartheta}\left(\hat{\vartheta}_{T'}-\vartheta\right)=\frac{2H \Gamma(H+1/2)\Gamma(3-2H)}{\Gamma(3/2-H)},\, H>1/2
$$
and 
$$
\lim_{T'\rightarrow \infty} T' \mathbf{E}_{\vartheta}\left(\hat{\vartheta}_{T'}-\vartheta\right)=1,\, H<1/2.
$$
Then the Delta method achieves the proofs..
\end{proof}

\subsection{Simulation of the estimated ruin probability}
In last part, we have verified that our simulation method for the function $g(s,t)$ is efficient for all $H\in (0,1)$.  Now, we use this method to verify the asymptotical properties of $\hat{\vartheta}_{T'}$ and the associated $\hat{\psi}(u)$. For the sake of simplify, we only check the case of $H>1/2$. To evidence the result of \eqref{eq: different distribution}, we now investigate
the asymptotic distribution of $\hat{\vartheta}_{T'}-\vartheta$ (given by \eqref{eq: different distribution}). Here, the chosen parameters are $H$=0.6, $\theta$=1.2 and $\sigma$=1 and we take $T'$=100. The histogram in the Figure 3 indicates that the normal approximation of the asymptotic distributions of $\hat{\vartheta}_{T'}-\vartheta$
is reasonable even with small $T'$.

\begin{center}
\resizebox{160mm}{80mm}{\includegraphics{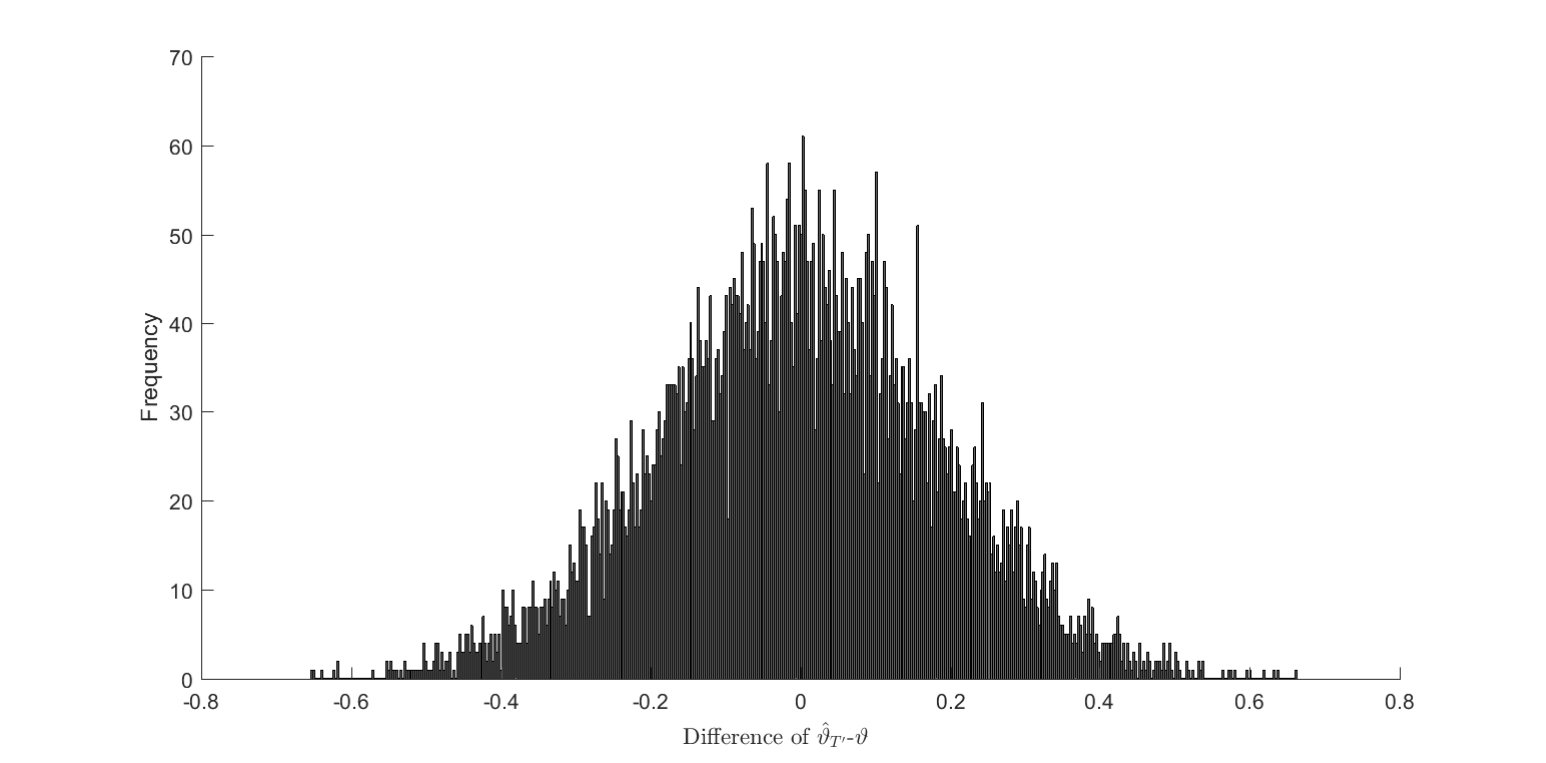}}
\\Fig.3. The distribution of the difference of $\hat{\vartheta}_{T'}-\vartheta$.
\end{center}

\begin{center}
\resizebox{160mm}{80mm}{\includegraphics{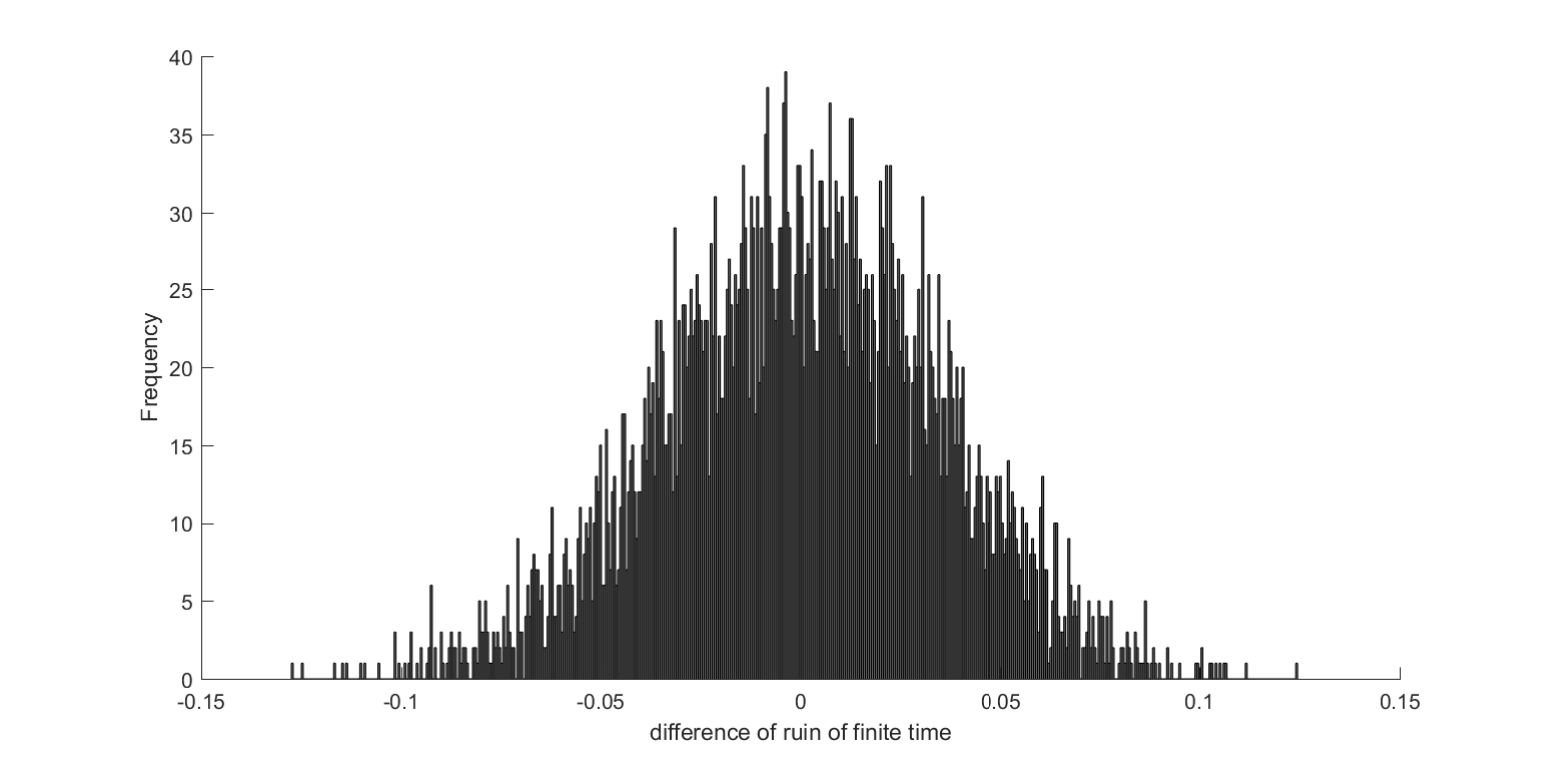}}
\\Fig.4. Difference of ruin.
\end{center}

In fact, the normal property of $\hat{\vartheta}_{T'}$ is nature. For the ultimate ruin probability, from eqref{finite ruin exact 1} when $a$ is large enough we can use $\psi(u,T)$ to replace $\psi(u)$, so we choose $T$=200, $u=2$, $H$=0.6, $\vartheta$=1.2. Then the ruin probability will be $\psi(u,T)$=0.0870. The variance of the right of \eqref{eq: bigger asym} will be 0.0011. The Figure 4 gives the normal distribution of the difference of the ruin probability with the variance of normal 0.0012 which is in line of our expectation.

\section{Proofs for the Main Results}\label{sec: Appendix}
\subsection{Proof for the Girsanov Formula: Theorem \ref{Girsanov formula}}
To prove the Girsanov formula, we only need to prove that the finite-dimensional distribution of $\xi$, with respect to the measure $\mathbf{P}_a$, are those of a mfBm with a drift $a$. We fix $t_i\in [0,T]$, $i=1,...n$ and put $\tau^{*}=(t_1,...t_n)$. We need to show
$$
\left(X_{t_1},...,X_{t_n}\right) \xlongequal  {\mathbf{P}_a}\mathcal{N}(a\tau,\Sigma)
$$
where $\Sigma=(\cov_0(\xi_{t_i},\xi_{t_j}))_i,j=1,...,n$.

Pick $\alpha^{*}=(\alpha_1,...\alpha_n)$ and note that
$$
\mathbf{E}_a \exp \left(\sum_{i=1}^n \alpha_i\xi_{t_i}\right)=\mathbf{E}_0\exp\left(\sum_{i=1}^n\alpha_i\xi_{t_i}+aM_T-\frac{a^2}{2}\langle M,M\rangle_T\right).
$$

With respect to the measure $\mathbf{P_0}$, the random variable $U=\sum_{i=1}^n\alpha_i\xi_{t_i}+aM_T$ is a Gaussian variable with the mean $\mathbf{E}_0U=0$ and the variance
$$
\mathbf{E}_0U^2=\mathbf{E}_0(aM_T^2)+\mathbf{E}_0\left(\sum_{i=1}^n\alpha_i\xi_{t_i}\right)^2+2\sum_{i=1}^n\cov_0(aM_T,\alpha_i\xi_{t_i}).
$$

We will calculate $\cov_0(M_t,\,\xi_s)$  for $s<t$. From the equation (3.18) in \cite{CCK}
\begin{equation*}
\begin{aligned} % requires amsmath; align* for no eq. number
\cov_0(M_t,\, \xi_s) & = \int_0^tg(r,t)\mathbf{1}_{[0,s]}(r)dr+\int_0^t\mathbf{1}_{[0,s]}(r)\frac{\partial}{\partial r}\int_0^t g(\ell, t)\frac{\partial}{\partial \ell} R(\ell, r) d\ell dr\\
& = \int_0^t\mathbf{1}_{[0,s]}(r)\left\{g(r,t)+\frac{\partial}{\partial r}\int_0^t g(\ell, t)\frac{\partial}{\partial \ell} R(\ell, r)d\ell \right\}dr\,.
\end{aligned}
\end{equation*}
From the equation \eqref{eq: integral big}, we have
$$
g(r,t)+\frac{\partial}{\partial r}\int_0^t g(\ell, t)\frac{\partial}{\partial \ell} R(\ell, r)d\ell=1,\, 
 $$
Thus $\cov_0(M_t,\, \xi_s)=s$ which implies that $\cov_0(aM_T,\,\alpha_i\xi_{t_i})=\alpha_ia t_i$. Hence, we have
$$
\mathbf{E}_0U^2=a^2\langle M, M\rangle_T+\alpha^{*}\Sigma\alpha+2\sum_{i=1}^n\alpha_ia t_i\,,
$$
when $U$ is Gaussian and
$$
\mathbf{E}_0\exp(U)=\exp\left(\frac{a^2\langle M, M\rangle_T+2 a \alpha^{*}\tau+\alpha^{*}\Sigma\alpha}{2}  \right)\,.
$$

As a consequence, we can show that
$$
\mathbf{E}_a\exp\left(\sum_{i=1}^n\alpha_i\xi_{t_i}\right)=\exp\left(a \alpha^{*}\tau+\frac{\alpha^{*}\Sigma\alpha}{2}\right)\,.
$$
\subsection{The performance of $\eta_a(v)$: Theorem \ref{upper bound} and \ref{asymptotic convergence}}
To get the upper bound of $\eta_a(v)$ in the Theorem \ref{upper bound}, let $\Phi(t)$ be the distribution function of the standard normal distribution $U\sim\mathcal{N}(0,1)$ and the $\phi(t)$ be the density function. By some basic calculations, we find
\begin{eqnarray*}
\mathbf{P}\left\{\sup_{s\leq t}\left(\xi_s+as\right)>v\right\} &\geq& \mathbf{P} \left\{ \xi_t+at>v\right\}=\mathbf{P}\left(\xi_t>v-at\right)\\
&=&\mathbf{P} \left(U>\frac{v-at}{\sqrt{t+t^2}}\right)=1-\Phi\left(\frac{v-at}{\sqrt{t+t^{2H}}} \right)\,.
\end{eqnarray*}

Since $\Phi(t)$ is a non-decreasing function and $t+t^{2H}\geq 2t^{(2H+1)/2}$, we have
$$
\mathbf{P}\left\{\sup_{s\leq t}\left(\xi_s+as\right)>u\right\}\geq 1-\Phi\left(\frac{v}{\sqrt{2}} t^{-(2H+1)/4}-\frac{a}{\sqrt{2}} t^{1-(2H+1)/4}\right)=1-\Phi\left(\frac{v}{\sqrt{2}} t^{-H'}-\frac{a}{\sqrt{2}} t^{1-H'}\right).
$$

Now, with the integrating by parts, we have
\begin{eqnarray*}
\mathbf{E}\eta_a^n(v)&=&\int_0^{\infty}\mathbf{P}\left(\eta_a(v)>t^{1/n}\right)dt\\
&=&\int_0^{\infty}\left(1- \mathbf{P}\left\{\sup_{s\leq t^{1/n}}\left(\xi_s+as\right)>v\right\}\right)dt\\
&\leq& \int_0^{\infty}\Phi \left(\frac{v}{\sqrt{2}} t^{-H'/n}-\frac{a}{\sqrt{2}} t^{(1-H')/n}\right)dt\\
&\leq&  \int_0^{\infty}   \left(\frac{vH't^{-H'/n}}{\sqrt{2}n}+\frac{(1-H^{'})at^{(1-H')/n})}{\sqrt{2}n}\right)\phi\left(\frac{u}{\sqrt{2}}t^{-H'/n}-\frac{a}{\sqrt{2}}t^{(1-H')/n}\right) dt.
\end{eqnarray*}
Finally, we divided the integral into two parts:
$$
\frac{vH'}{\sqrt{2}n\sqrt{2\pi}}\int_0^{\infty}t^{-H'/n}\exp\left\{-\frac{1}{2}\left(\frac{v}{\sqrt{2}}t^{-H'/n}-\frac{a}{\sqrt{2}}t^{(1-H')/n}\right)^2\right\}dt
$$
and
$$
\frac{a(1-H')}{\sqrt{2}n\sqrt{2\pi}}\int_0^{\infty}t^{(1-H')/n}\exp\left\{-\frac{1}{2}\left(\frac{v}{\sqrt{2}}t^{-H'/n}-\frac{a}{\sqrt{2}}t^{(1-H')/n}\right)^2\right\}dt
$$
and substituting $t=s^{n/(n-H')}$ and $t=s^{n/(n+1-H')}$,  respectively, we obtain the upper bound.

Now we will prove the Theorem \ref{asymptotic convergence}, with the Lemma 1 in \cite{Michna98b} we have for $H>\frac{1}{2}$, $n\in \mathbb{N}$ and $u\rightarrow \infty$
\begin{equation}\label{Lnu}
L_n(u,H,a)\sim \sqrt{2\pi}(n-H)a^{-n}u^{n-1}.
\end{equation}
Then in the theorem \ref{upper bound}, we have
\begin{eqnarray*}
&&\frac{1}{\sqrt{2\pi}}\left(\frac{vH'}{\sqrt{2}(n-H')}L_n(v/\sqrt{2},H',a/\sqrt{2})+\frac{(1-H')a}{\sqrt{2}(n+1-H')} L_{n+1}(v/\sqrt{2},H',a/\sqrt{2}) \right)\\
&\sim& \left(\frac{v}{\sqrt{2}}\right)^n\left(\frac{a}{\sqrt{2}}\right)^{-n}\,.
\end{eqnarray*}
As a consequence, we have
$$
\lim_{v\rightarrow \infty}\mathbf{E}\eta_a^n(v)\left(\frac{v}{\sqrt{2}}\right)^{-n} \leq \left(\frac{a}{\sqrt{2}}\right)^{-n}.
$$
Now we will find  the lower bound. Using the Slepian's inequality as in \cite{Michna98b} and Lemma 1 in \cite{Michna98a}, we can have
$$
\mathbf{P}\left\{\sup_{s\leq t}(\xi_s+as)>v\right\}\leq \mathbf{P}\left\{\sup_{s\leq t}(\bar{W}(s+s^{2H}) +as)>v\right\}\leq \mathbf{P} \left\{\sup_{s\leq t}\bar{W}(\frac{1}{2}(s+s^{2H})+\frac{1}{\sqrt{2}}as)>\frac{v}{\sqrt{2}} \right\}\,,
$$
where $\bar{W}(t)$ is a standard Brownian motion.

Let $f(t)=\frac{1}{2}(t+t^{2H})$ and we denote  $\rho(u)$ as
$$
\rho(v)=\inf\{t>0: \bar{W}(t)+\frac{a}{\sqrt{2}}f^{-1}(t)>\frac{v}{\sqrt{2}}\}\,.
$$

Then, we have
\begin{eqnarray*}
\mathbf{P} \left\{\sup_{s\geq t}\bar{W}(\frac{1}{2}(s+s^{2H})+\frac{1}{\sqrt{2}}as)>\frac{v}{\sqrt{2}}\right\}&=&\mathbf{P}\left\{\sup_{s\leq f(t)}(\bar{W}(s)+\frac{a}{\sqrt{2}}f^{-1}(s))>\frac{v}{\sqrt{2}}\right\}\\
&=&\mathbf{P}\{f^{-1}(\rho(v))\leq t\}\,,
\end{eqnarray*}
with the barrier $\frac{v}{\sqrt{2}}-\frac{a}{\sqrt{2}}f^{-1}(t)$.

Indeed, we can see that the lower bound of the density of $\rho(u)$ is
\begin{eqnarray*}
p(t)&=&t^{-3/2}(\frac{v}{\sqrt{2}}-\frac{a}{\sqrt{2}}f^{-1}(t)-t[\frac{v}{\sqrt{2}}-\frac{a}{\sqrt{2}}f^{-1}(t)]')\phi(t^{-1/2}(\frac{v}{\sqrt{2}}-\frac{a}{\sqrt{2}}f^{-1}(t)))\\
&=&t^{-3/2}\left(\frac{v}{\sqrt{2}}-\frac{a}{\sqrt{2}}f^{-1}(t)+\frac{a}{\sqrt{2}}t[f^{-1}(t)]'\right)\phi(t^{-1/2}(\frac{v}{\sqrt{2}}-\frac{a}{\sqrt{2}}f^{-1}(t)))\\
&=& t^{-3/2}\left(\frac{v}{\sqrt{2}}-\frac{a}{\sqrt{2}}f^{-1}(t)+\frac{2a}{\sqrt{2}}t\left[\frac{1}{1+2H(f^{-1}(t)^{2H-1})}      \right]\right)\phi(t^{-1/2}(\frac{v}{\sqrt{2}}-\frac{a}{\sqrt{2}}f^{-1}(t)))\,.
\end{eqnarray*}

Now we construct a function $F(x)$, which is defined by
$$
F(x)=x^nt^{-3/2}\left(\frac{v}{\sqrt{2}}-\frac{a}{\sqrt{2}}x+\frac{2a}{\sqrt{2}}t\left[\frac{1}{1+2Hx^{2H-1}} \right]\right)\phi(t^{-1/2}(\frac{v}{\sqrt{2}}-\frac{a}{\sqrt{2}}x))\,.
$$

Thus $F(x)$ is a non-increasing function when $x\leq v/a$ and a non-decreasing function when $x\geq v/a$. In fact for every $t$ fixed, we have $f^{-1}(t)\leq t^{1/(2H')}$ if $H'=\frac{2H+1}{4}$. We can also have $f^{-1}(t)\geq t^{1/(2H)}$ if $t\geq 1$. Thus, if $u$ is a fixed number, then for $t\leq f(v/a)$, $F(x)$ is non-increasing and non-decreasing for $t\geq f(v/a)$.

Moreover, a standard calculation yields
$$
\mathbf{E}\eta_a^n(v)\geq \mathbf{E}\left[f^{-1}(\rho(v))\right]^n
\geq \int_0^{\infty}\left[f^{-1}(t)\right]^np(t)dt\,.
$$

Now with the properties of $F(x)$, we can see that
\begin{eqnarray*}
\mathbf{E}\eta_a^n(v)&\geq &\int_0^{f(v/a)}t^{n/(2H')-3/2}\left(\frac{v}{\sqrt{2}}-\left(1-\frac{1}{2H'}\right)\frac{a}{\sqrt{2}}t^{1/(2H')}\right)\phi\left(\frac{v}{\sqrt{2}}-\frac{a}{\sqrt{2}}t^{(1-H')/(2H')}\right)\\\
&+&\int_{f(v/a)}^{\infty}t^{n/(2H')-3/2}\left(\frac{v}{\sqrt{2}}-\left(1-\frac{1}{2H}\right)\frac{a}{\sqrt{2}}t^{1/(2H)}\right)\phi\left(\frac{v}{\sqrt{2}}-\frac{a}{\sqrt{2}}t^{(1-H')/(2H)}\right)\,.
\end{eqnarray*}

Using Lemma 1 in \cite{Michna98b}, we have
\begin{equation}\label{lower bound}
\lim_{v\rightarrow \infty} \frac{\int_0^{f(v/a)}t^{n/(2H')-3/2}\left(\frac{v}{\sqrt{2}}-\left(1-\frac{1}{2H'}\right)
\frac{a}{\sqrt{2}}t^{1/(2H')}\right)\phi\left(\frac{v}{\sqrt{2}}
-\frac{a}{\sqrt{2}}t^{(1-H')/(2H')}\right)}{\left(\frac{v}{\sqrt{2}}\right)^{n}}
=\left(\frac{a}{\sqrt{2}}\right)^{-n}\,.
\end{equation}
$$
\lim_{v\rightarrow \infty}\mathbf{E}\eta_a^n(v)\left(\frac{v}{\sqrt{2}}\right)^{-n} \geq  \left(\frac{a}{\sqrt{2}}\right)^{-n}.
$$
Then \eqref{limit expectation} is satisfied. At the same time \eqref{no expectation} or \eqref{Ln convergence} is satisfied.

\end{document}